\documentclass[a4paper,twoside,11pt,reqno]{amsart}
\usepackage{fullpage}
\usepackage[T1]{fontenc}
\usepackage[utf8]{inputenc}

\usepackage{hyperref}
\usepackage[slantedGreeks, partialup, noDcommand]{kpfonts}
\usepackage{graphicx}
\usepackage[mathscr]{euscript}
\usepackage{caption}
\usepackage{tikz}
\usepackage{appendix}

\usetikzlibrary{trees}

\hypersetup{ocgcolorlinks=true,allcolors=testc}
\hypersetup{
	colorlinks   = true,
	citecolor    = black
}
\hypersetup{linkcolor=black}
\hypersetup{urlcolor=black}

\newtheorem{theorem}{Theorem}
\newtheorem{lemma}{Lemma}

\newtheorem{proposition}{Proposition}
\newtheorem{corollary}[theorem]{Corollary}

\newtheorem{remark}[theorem]{Remark}
\newtheorem*{remark*}{Remark}

\newtheorem*{definition*}{Definition}

\theoremstyle{definition}

\theoremstyle{remark}

\renewcommand{\epsilon}{\varepsilon}

\renewcommand{\phi}{\varphi}
\newcommand{\R}{\mathbb{R}}

\newcommand{\Z}{\mathbb{Z}}

\DeclareMathOperator{\im}{Im}

\DeclareMathOperator{\supp}{supp}
\DeclareMathOperator{\sgn}{sgn}

\DeclareMathOperator{\E}{\mathbb{E}}

\def \R {\mathbb{R}}
\def \Q {\mathbb{Q}}
\def \Z {\mathbb{Z}}

\def \tran {\mathsf{T}}

\def\bal{\begin{align*}}
\def\eal{\end{align*}}

\DeclareMathOperator{\rank}{rank}
\DeclareMathOperator{\diag}{diag}
\DeclareMathOperator{\Cok}{Cok}
\DeclareMathOperator{\Cons}{Cons}

\author{Xinyuan Xie}
\address{(X.~X.) Department of Mathematics, University of California, Irvine, CA 92697, USA}
\email{xinyuax7@uci.edu}

\title{Boolean threshold functions, neuron capacity, and memory retrieval}

\thanks{}

\begin{document}

    \maketitle

\begin{abstract}
\setlength{\abovedisplayskip}{4pt}
\setlength{\belowdisplayskip}{4pt}
\setlength{\abovedisplayshortskip}{3pt}
\setlength{\belowdisplayshortskip}{3pt}
How much information can a single neuron remember? How many memories can neural networks retrieve without creating false memories? These questions are related to a basic question: how many Boolean threshold functions $f(x)=\operatorname{sgn}(a_0+\langle a,x\rangle)$, $x\in\{-1,1\}^n$, are there? In this paper, we show that the number $T_n$ of distinct Boolean threshold functions is
\[
T_n=2\binom{2^n-1}{n}\bigl(1+O(n^{-99})\bigr).
\]
Equivalently, the capacity of a single threshold neuron is $n^2-\log_2(n!)+1+O(n^{-99})$ bits, improving the $O(n)$ error term in the result of Kahn--Koml\'os--Szemer\'edi to $O(n^{-99})$. To prove this, we show that, for $1\le r\le n-1$, and $v_1,\ldots,v_r$ are chosen at random from $\{-1,1\}^n$,
\[
\mathbb P\!\left\{
\langle v_1,\ldots,v_r\rangle\cap\{-1,1\}^n
=\{\pm v_1,\ldots,\pm v_r\}
\right\}
=1-O(n^{-99}).
\]
In the context of the Kanter--Sompolinsky Hamiltonian for memory retrieval, this identifies $r=n-1$ as a sharp threshold, at which, for almost every collection of $r$ memories, the only ground states are these memories and their negatives, confirming a weaker form of the Kalai--Linial--Odlyzko conjecture. It also settles a recent open problem posed by M.~Anthony on the specification number of Boolean threshold functions. In addition, we show that, for every $1\le r\le n-1$,
\[
\mathbb P\{v_1,\ldots,v_r\text{ are linearly dependent}\}
=2\binom r2\,2^{-n}+O\!\left(2^{-n}e^{-cn}\right),
\]
confirming a conjecture of Kahn--Koml\'os--Szemer\'edi.
\end{abstract}

\bigskip

{\footnotesize

}

\section{Introduction}\label{sec:intro}

\subsection{Boolean threshold functions}\label{sec:intro-count}
A Boolean function $f:\{-1,1\}^n\to\{-1,1\}$ is called a \emph{Boolean threshold function}, or a \emph{linear threshold function}, if it admits the representation
\begin{equation}\label{eq:ltf}
f(x)=\sgn(a_0+a_1x_1+\cdots+a_nx_n),
\end{equation}
for some $a_0,\ldots,a_n\in\R$.

These functions have been used to model neural activity back to McCulloch and Pitts~\cite{MP}. Since then, threshold functions became building blocks and the first basic object for theoretical analysis in neural networks~\cite{Rosenblatt,Hopfield}, pattern recognition~\cite{Cover}, learning theory~\cite{Bylander,KlivansODonnellServedio}, and many other areas, see the references in~\cite{BVthreshold} for more applications. 

Surprisingly, the number $T_n$ of distinct Boolean threshold functions on $\{-1,1\}^n$, such a basic quantity, remains unknown. There has been extensive activity in determining this quantity asymptotically since the 1960s. The classical upper bound
\begin{equation}\label{eq:classical-count-upper}
T_n\le 2\sum_{j=0}^n\binom{2^n-1}{j},
\end{equation}
was obtained in the early 1960s by several authors~\cite{Cameron,Joseph,Winder}, and their attributions can be found in ~\cite{Cover}. For the lower bound side, Muroga~\cite{Muroga} proved the lower bound $T_n\ge 2^{0.33048n^2}$, and Yajima and Ibaraki~\cite{YI} improved it to $T_n\ge 2^{n(n-1)/2+8}$ for $n\ge6$. Using Odlyzko's estimate for spans of random sign vectors~\cite{Odlyzko}, Zuev~\cite{Zuev} obtained
\[
T_n\ge 2^{\,n^2-10n^2/\ln n+O(n\ln n)}.
\]
The last rigorous improvement before the present work is due to Kahn, Koml\'os, and Szemer\'edi~\cite[Section~4]{KKS}, who sharpened the lower bound to
\begin{equation}\label{eq:kks-count}
T_n\ge 2^{\,n^2-n\log_2 n-O(n)}.
\end{equation}
However, the bounds~\eqref{eq:classical-count-upper} and~\eqref{eq:kks-count} still leave an exponential multiplicative uncertainty in $T_n$. More recently, Irmatov~\cite{Irmatov} claimed the asymptotic formula $T_n\sim 2\binom{2^n-1}{n}$. As noted in~\cite{JSSsing}, experts have identified unresolved issues in that paper, and we do not see how the current arguments are able to derive such asymptotics. In the next theorem, we show that this asymptotics is indeed correct.

\begin{samepage}
\begin{theorem}[Counting Boolean threshold functions]\label{thm:count}
For every integer $n\ge1$, the number $T_n$ of Boolean threshold functions on $\{-1,1\}^n$ satisfies
\begin{equation}\label{eq:count}
T_n=2\binom{2^n-1}{n}\bigl(1+O(n^{-99})\bigr),
\end{equation}
where the implicit constant is absolute.
\end{theorem}
\end{samepage}

The upper bound follows from~\eqref{eq:classical-count-upper}. The lower bound is implied by a geometric problem on the span of several Rademacher random vectors, as observed by several authors~\cite{ZuevGeometry,KKS,Irmatov1996}. We formulate this problem in Section~\ref{sec:intro-span} and derive the lower bound from Theorem~\ref{thm:span} in Section~\ref{sec:arrangements}.

\subsection{Capacity of a threshold neuron}\label{sec:intro-capacity}
As modeled by McCulloch and Pitts~\cite{MP}, a Boolean threshold function is corresponding to a threshold neuron: $a_1,\ldots,a_n$ are the synaptic weights, and the activation status of the neuron is determined by whether $a_1x_1+\cdots+a_nx_n$ is above or below $-a_0$. A basic question is how much information can be represented by varying these parameters, i.e., what is the capacity of a threshold neuron. 

The questions about capacity have been addressed for various important architecture in the development of neural network. For example, the information and storage capacity of Hopfield's network was determined to relative precise accuracy ~\cite{AbuMostafaStJacques, McElieceEtAl}. As the building block of Hopfield's network, the capacity of a threshold neuron,  although is implicit in the many previous counting results \cite{Cameron,Joseph,Winder,Cover,Muroga,YI,Zuev,KKS}, was first formulated in the literature in Baldi-Vershynin~\cite{BVcapacity}: the capacity of a class of functions is the logarithm base two of the number of functions it can implement, and for the single threshold neuron on $\{-1,1\}^n$, its capacity is precisely $\log_2T_n$ bits. The latest best known result before the present work is due to Kahn, Koml\'os, and Szemer\'edi~\cite[Section~4]{KKS}, whose bounds imply
\[
\log_2 T_n
=
n^2-\log_2(n!)+1+O(n).
\]
Our result reduces the $O(n)$ error term to $O(n^{-99})$. 

\begin{samepage}
\begin{corollary}[Capacity of a threshold neuron]\label{thm:capacity}
For every integer $n\ge1$, the capacity $\log_2T_n$ of a threshold neuron on $\{-1,1\}^n$ satisfies
\begin{equation}\label{eq:single-capacity}
\log_2T_n
=
n^2-\log_2(n!)+1+O(n^{-99}),
\end{equation}
where the implicit constant is absolute.
\end{corollary}
\end{samepage}

Corollary~\ref{thm:capacity} follows directly from Theorem~\ref{thm:count} by taking binary logarithms, see Section~\ref{sec:capacity-proof}.

\subsection{Spans of random sign vectors}\label{sec:intro-span}
In this section, we consider the following geometric problem: for $v_1,\ldots,v_r$ chosen independently and uniformly from $\{-1,1\}^n$, what is the probability such that the event 
\[
\langle v_1,\ldots,v_r\rangle\cap\{-1,1\}^n
=\{\pm v_1,\ldots,\pm v_r\} \quad \text{holds ?}
\]

Notice that, by Koml\'os's result~\cite{Komlos01}, this probability tends to zero when $r=n$. On the other hand, Odlyzko~\cite{Odlyzko} proved that it tends to one uniformly for $r\le n-10n/\log n$. These results formed the Kalai--Linial--Odlyzko conjecture: the largest $r$ for which the event above holds with probability $1-o(1)$ is $r=n-1$, and the failure probability is asymptotic to
\[
4\binom r3\left(\frac34\right)^n.
\]
Kahn, Koml\'os, and Szemer\'edi~\cite[Corollary~4(b)]{KKS} later extended Odlyzko's bound to $r\le n-C$ for some absolute $C>0$. More recently, Irmatov~\cite{IrmatovSpans} claimed this sharp asymptotic for all $r\le n-1$. However, several issues in the argument remain unresolved, and we do not see how the current proof yields the stated asymptotic. In the next theorem, we show that the transition indeed occurs at $r=n-1$, with the error term weaker than the original conjecture.

\begin{samepage}
\begin{theorem}[Span of random sign vectors]\label{thm:span}
Let $1\le r\le n-1$, and let $v_1,\ldots,v_r$ be independent random vectors uniformly distributed in $\{-1,1\}^n$. For all sufficiently large $n$, with probability at least $1-Cn^{-99}$, these vectors are linearly independent and
\begin{equation}\label{eq:span}
\langle v_1,\ldots,v_r\rangle\cap \{-1,1\}^n
=\{\pm v_1,\ldots,\pm v_r\},
\end{equation}
where $C>0$ is an absolute constant.
\end{theorem}
\end{samepage}

\begin{remark}
The exponent $99$ here and in all the other related results can be replaced by any fixed $A>0$, as can be seen from the proof of Theorem~\ref{thm:span} in Section~\ref{sec:span-proof}, with the implicit constant depending on~$A$.
\end{remark}

This is a main technical difficulty of the present paper, and the proof has two main steps. The first step, Proposition~\ref{prop:reduction}, reduces the problem to an integral cokernel estimate. A concentration argument first shows that the relevant integer relations have many exponentially large coefficients. We then assign weights inversely proportional to the number of additional vertices of the cube, and use column symmetry to reduce the problem to an estimate of the last coefficient $a_n$. Since this coefficient divides the normalized greatest common divisor of the maximal minors, it remains to bound the probability that the cokernel of an $(n-2)\times(n-1)$ matrix with independent uniform $0/1$ entries has large order. The second step estimates this cokernel through bounds for higher moments, using Maples's comparison theorem~\cite[Theorem~1.2]{Maples}, Nguyen--Wood's rank estimate modulo a prime~\cite[Sections~6--7]{NW}, and inverse Littlewood--Offord theory. Together, these ingredients give the error term in Theorem~\ref{thm:span}. For a weaker $n^{-1+\varepsilon}$ error bound, the proof can be simplified by replacing the estimate of higher moments with the bound for the first moment together with the estimates already established in~\cite[Propositions~2.2--2.3]{NW}.

\subsection{Memory retrieval}\label{sec:intro-memory}
In a pioneering work, Hopfield~\cite{Hopfield}, based on the threshold model introduced by McCulloch and Pitts~\cite{MP}, described an algorithm for storing memories, i.e., points in the Boolean hypercube, in the corresponding neural network, with the feature that a stored memory can be retrieved from different input patterns similar to it. Personnaz, Guyon, and Dreyfus~\cite{PGD} subsequently proposed, for the same neural network, a different algorithm for storing and retrieving information, called the \emph{projection rule}. For this projection rule, Kanter and Sompolinsky~\cite{KS} proposed a Hamiltonian to interpret its mechanism for memory retrieval. In this setting, the memories to be stored and retrieved are linearly independent vectors $v_1,\ldots,v_r\in\{-1,1\}^n$. The projection rule works as follows. Suppose these memories have been stored, which essentially means that one forms
\begin{equation}\label{eq:projection}
V=(v_1,\ldots,v_r),\qquad
P_V=V(V^{\tran}V)^{-1}V^{\tran}.
\end{equation}
Given an input $x\in\{-1,1\}^n$ similar to a stored memory $v_i$, one computes $\sgn(P_Vx)$ and compares it with $v_i$. If $\sgn(P_Vx)=v_i$ (or $v_i$ is reached after finitely many iterations), then the memory $v_i$ is retrieved successfully; otherwise it is not. For more background on the development of Hopfield networks and the projection rule, see~\cite{YampolskayaMehta}.

The Hamiltonian proposed by Kanter and Sompolinsky to analyze this rule, in an equivalent form, is
\begin{equation}\label{eq:energy-intro}
H_{\mathrm{KS}}(x)=-\frac12x^{\tran}P_Vx
=-\frac n2+\frac12\|(I-P_V)x\|_2^2,
\qquad x\in\{-1,1\}^n.
\end{equation}
One can easily see that the minimizers of $H_{\mathrm{KS}}$ are exactly the vertices from $\{-1,1\}^n$ in $\langle v_1,\ldots,v_r\rangle$. These minimizers are called \emph{ground states}, and in the thermodynamic framework of Kanter and Sompolinsky they have a distinguished role in the analysis of the memory retrieval mechanism. This leads to the question of how many memories can be stored without creating ground states other than $\{\pm v_1,\ldots,\pm v_r\}$, i.e., the stored memories and their negatives. For most collections of $r$ memories, Odlyzko's result~\cite{Odlyzko} shows that the ground states are exactly $\{\pm v_1,\ldots,\pm v_r\}$ whenever $r\le n-10n/\log n$, and Kahn, Koml\'os, and Szemer\'edi~\cite[Corollary~4(b)]{KKS} improved this range to $r\le n-C$ for some absolute constant $C>0$. Theorem~\ref{thm:span} therefore settles the sharp transition point at $r=n-1$: $n-1$ is the largest possible number of memories such that, with probability $1-o(1)$, the ground states are exactly $\{\pm v_1,\ldots,\pm v_r\}$:

\begin{samepage}
\begin{corollary}[Threshold for number of ground states]\label{cor:memory}
Let $1\le r\le n-1$, and choose $v_1,\ldots,v_r$ independently and uniformly from $\{-1,1\}^n$. For all sufficiently large $n$, with probability at least $1-Cn^{-99}$, the patterns are linearly independent and
\[
\operatorname*{arg\,min}_{x\in \{-1,1\}^n}H_{\mathrm{KS}}(x)
=\{\pm v_1,\ldots,\pm v_r\}.
\]
\end{corollary}
\end{samepage}

\subsection{Specification numbers}\label{sec:intro-spec}
For a Boolean threshold function $f$ on $\{-1,1\}^n$, let $\sigma_n(f)$ be the smallest size of a set of labelled examples $(x,f(x))$ that distinguishes $f$ from every other Boolean threshold function on $\{-1,1\}^n$. Define
\begin{equation}\label{eq:spec-definition}
\sigma_n=\frac1{T_n}\sum_f\sigma_n(f),
\end{equation}
where the sum is over all Boolean threshold functions on $\{-1,1\}^n$.

Martin Anthony~\cite[Theorem~6.1]{Anthony} proved $n+1\le\sigma_n\le2n$. In Problem~6.2 of that paper, he asked whether $\sigma_n/(n+1)$ converges and, if so, what its limit is. We settle this problem by proving that the limit is $2$.

\begin{samepage}
\begin{theorem}[Average specification number]\label{thm:spec}
For all sufficiently large $n$,
\begin{equation}\label{eq:spec}
2n-Cn^{-98}\le\sigma_n\le2n.
\end{equation}
In particular, $\displaystyle\lim_{n\to\infty}\sigma_n/(n+1)=2$.
\end{theorem}
\end{samepage}

Anthony~\cite[Theorem~3.3]{Anthony} identifies $\sigma_n(f)$ with the number of facets of the corresponding chamber closure. Averaging over all chambers and applying the chamber--facet incidence formula~\cite[Lemma~5.1]{Anthony} reduces the problem to estimating the numbers of regions of the restricted arrangements. Theorem~\ref{thm:spec} then follows from the estimate~\eqref{eq:spec-restricted-lower} and the full proof is given in Section~\ref{sec:spec-proof}.

\subsection{The linear dependence of Rademacher random vectors}\label{sec:intro-dependence}
In this section, we study the following problem: for $v_1,\ldots,v_r$ chosen independently and uniformly from $\{-1,1\}^n$, what is the probability that these random vectors are linearly dependent? When $r=n$, this is equivalent to determining the probability that the corresponding $n\times n$ Rademacher random matrix is singular. Koml\'os first proved that this probability tends to zero~\cite{Komlos01} and later obtained the bound $O(n^{-1/2})$~\cite[Section~1.1]{KKS}. Since then, this problem has been extensively investigated and the upper bounds have continued to improve~\cite{KKS,TaoVu2006,TaoVu2007,BourgainVuWood,Tikhomirov}, and the current best result is due to Tikhomirov~\cite{Tikhomirov}, who proved
\[
\mathbb P\{M_n\text{ is singular}\}=(1/2+o(1))^n,
\]
which captures the probability at the correct exponential scale, whereas the conjectured finer asymptotic
\begin{equation}\label{eq:square-singularity-conjecture}
\mathbb P\{M_n\text{ is singular}\}
=(1+o(1))\,\frac{n^2}{2^{n-1}}
\end{equation}
still remains open.

However, if we consider $r\le n-C$, where $C>0$ is an absolute constant, an asymptotic formula with this level of precision was already obtained by Kahn, Koml\'os, and Szemer\'edi. More precisely, Kahn, Koml\'os, and Szemer\'edi~\cite[Corollary~4(a)]{KKS} proved that, for independent uniform vectors $v_1,\ldots,v_r\in\{-1,1\}^n$ with $2\le r\le n-C$,
\begin{equation}\label{eq:kks-dependence}
\mathbb P\{v_1,\ldots,v_r\text{ are linearly dependent}\}
=(1+o(1))\,2\binom r2\,2^{-n}.
\end{equation}
They conjectured that~\eqref{eq:kks-dependence} remains valid all the way to the endpoint $r=n-1$. The next theorem settles this conjecture and gives an exponentially smaller error term.

\begin{samepage}
\begin{theorem}[Linear dependence]\label{thm:independence}
Let $1\le r\le n-1$, and let $v_1,\ldots,v_r$ be independent uniform random vectors in $\{-1,1\}^n$. Then
\begin{equation}\label{eq:independence}
\mathbb P\{v_1,\ldots,v_r\text{ are linearly dependent}\}
=2\binom r2\,2^{-n}+O(2^{-n}e^{-cn}),
\end{equation}
where $c>0$ and the implied constant are absolute.
\end{theorem}
\end{samepage}

A main difficulty in the proof is to show that dependence relations involving more than two vectors have probability $O(2^{-n}e^{-cn})$ at the endpoint $r=n-1$. We divide these relations according to their concentration. Relations with small concentration are controlled directly by the subspace concentration bound of Kahn, Koml\'os, and Szemer\'edi~\cite[Lemma~2]{KKS}. For large concentration, the corank estimate of Jain, Sah, and Sawhney~\cite{JSSrank} and their analysis of almost-constant null vectors~\cite{JSSsing} reduce the problem to estimating the probability that the last row annihilates a null vector determined by the preceding rows, conditional on the last row lying in a nearly balanced set of sign vectors. We adapt their argument using symmetry under changes of column signs to show that this conditioning does not change the probability of the relevant event, so their conditional concentration estimate yields the required exponentially smaller bound for nontrivial relations. The proof is provided in Section ~\ref{sec:dependence}.

\begin{remark}
 Notice that if an asymptotic formula with the same level of precision could be obtained for $r=n$, it would settle the conjectured bound for the probability that a random Rademacher matrix is singular ~\eqref{eq:square-singularity-conjecture}.   
\end{remark}

\subsection{Capabilities of frontier reasoning LLMs}
It is striking that neural network models have advanced to the point where frontier large language models can themselves help resolve theoretical questions concerning the very basic objects from which the field began.
Here we record two such instances, the first one concerning Theorem~\ref{thm:count}, although the proof yields only a non-explicit o(1) error term, and the second one concerning Theorem~\ref{thm:independence}.

\medskip
\noindent\textit{A weaker form of Theorem~\ref{thm:count}.}
\par\noindent
\url{https://chatgpt.com/share/6ab4fc10-80e0-83ea-a8fd-b3165dd6d57d}

\medskip
\noindent\textit{A proof of Theorem~\ref{thm:independence}.}
\par\noindent
\url{https://chatgpt.com/share/6ab4fed2-e0c4-83ea-986a-50c61710df9d}

\section*{Acknowledgments}
The author is very grateful to Paata Ivanisvili for introducing him to the problem of counting Boolean threshold functions, and to Paata Ivanisvili and Roman Vershynin for helpful discussions on related questions and techniques. He is also very grateful to Haonan Zhang for asking about the importance of Boolean threshold functions and for helpful discussions.

\section{Proofs}\label{sec:proofs}

\subsection{Proof of Theorem~\ref{thm:count}}\label{sec:arrangements}
For a finite set $S\subset\R^d\setminus\{0\}$, two parameter vectors whose defining linear forms do not vanish on $S$ give the same sign function exactly when they lie on the same side of every hyperplane $x^\perp$, $x\in S$, and hence in the same region of the arrangement. This gives the following correspondence, as recorded in~\cite[Section~2]{BVthreshold} and~\cite[Section~3.2]{VershyninLowFrequencies}.

\begin{lemma}[Threshold functions and regions; {\cite[Lemma~2.1]{BVthreshold}}]\label{lem:threshold-regions}
For a finite set $S\subset\R^d\setminus\{0\}$, the number of homogeneous linear threshold functions on $S$ equals the number of regions of the arrangement $\{x^\perp:x\in S\}$.
\end{lemma}

Write $r(\mathcal A)$ for the number of regions of an arrangement $\mathcal A$. An \emph{intersection subspace} is a nonempty intersection of hyperplanes from $\mathcal A$. We also include the ambient space, which corresponds to the empty subfamily.

\begin{samepage}
\begin{lemma}[Counting regions of hyperplane arrangements; {\cite[Lemma~2.2(1),(2),(4)]{BVthreshold}}]\label{lem:arrangement-bounds}
Let $\mathcal A$ be an arrangement of $p\ge d$ affine hyperplanes in $\R^d$. Then
\begin{enumerate}
\item $r(\mathcal A)\le\displaystyle\sum_{j=0}^{d}\binom pj$;
\item $r(\mathcal A)$ is at least the number of intersection subspaces of $\mathcal A$;
\item if $\mathcal A$ is central, then
\[
r(\mathcal A)\le2\sum_{j=0}^{d-1}\binom{p-1}{j}.
\]
\end{enumerate}
\end{lemma}
\end{samepage}

Put
\[
E_n=\{(1,x):x\in\{-1,1\}^n\},\qquad
\mathcal A_n=\{e^\perp:e\in E_n\}.
\]
Lemma~\ref{lem:threshold-regions} gives $T_n=r(\mathcal A_n)$. Fix $e_0\in E_n$ and define the affine section and its induced arrangement by
\begin{equation}\label{eq:decone}
P_0=\{z:e_0\cdot z=1\},\qquad
\mathcal D_n=\{P_0\cap e^\perp:e\in E_n\setminus\{e_0\}\}.
\end{equation}
Intersecting the regions of $\mathcal A_n$ on the positive side of $e_0^\perp$ with $P_0$ gives a bijection with the regions of $\mathcal D_n$. Negation pairs these regions of $\mathcal A_n$ with those on the negative side, so
\begin{equation}\label{eq:decone-count}
T_n=r(\mathcal A_n)=2r(\mathcal D_n).
\end{equation}

\begin{samepage}
\begin{lemma}[Intersection subspaces of the threshold arrangement]\label{lem:intersections}
For all sufficiently large $n$, at least $\binom{2^n-1}{n}(1-Cn^{-99})$ of the subsets $S\subset E_n\setminus\{e_0\}$ of size $n$ are linearly independent and satisfy
\[
E_n\cap\langle S\rangle=S.
\]
Their intersections $P_0\cap\bigcap_{e\in S}e^\perp$ are distinct nonempty intersection subspaces of $\mathcal D_n$, each contained in exactly the hyperplanes indexed by $S$.
\end{lemma}
\end{samepage}

\begin{proof}
Apply Theorem~\ref{thm:span} in dimension $n+1$ to $n$ independent uniform sign vectors. Multiplying each vector by its first coordinate gives independent uniform points of $E_n$ without changing their span. A repetition or the selection of $e_0$ has probability $O(n^2 2^{-n})$. Conditioning on the absence of these events therefore preserves the error $O(n^{-99})$. After forgetting the order, the conditioned sample is uniform over the subsets of $E_n\setminus\{e_0\}$ of size $n$, which gives the stated count.

Since $E_n\cap\langle S\rangle=S$ and $e_0\notin S$, adjoining $e_0$ to such a set $S$ gives a basis of $\R^{n+1}$. Consequently, $\bigcap_{e\in S}e^\perp=\langle S\rangle^\perp$ is a line not contained in $e_0^\perp$ and therefore meets $P_0$ in exactly one point. A hyperplane $f^\perp$, $f\in E_n$, contains this point if and only if $f\in\langle S\rangle$, or equivalently $f\in S$. Thus the hyperplanes containing the point determine $S$, so distinct subsets give distinct intersection subspaces of $\mathcal D_n$.
\end{proof}

\begin{proof}[Proof of Theorem~\ref{thm:count}]
The intersection count in Lemma~\ref{lem:intersections} and the region bounds in Lemma~\ref{lem:arrangement-bounds} give
\begin{align}\label{eq:arrangement-count}
\binom{2^n-1}{n}(1-Cn^{-99})
&\le r(\mathcal D_n)\le\sum_{j=0}^n\binom{2^n-1}{j}\notag\\
&=\binom{2^n-1}{n}\bigl(1+O(n2^{-n})\bigr).
\end{align}
Since $T_n=2r(\mathcal D_n)$, this proves~\eqref{eq:count}.
\end{proof}

\begin{samepage}
\subsection{Proof of Corollary~\ref{thm:capacity}}\label{sec:capacity-proof}
\begin{proof}
By Theorem~\ref{thm:count},
\[
\log_2 T_n
=1+\log_2\binom{2^n-1}{n}+O(n^{-99}).
\]
To evaluate the binomial term, write
\[
\log_2\binom{2^n-1}{n}
=n^2-\log_2(n!)+\sum_{j=1}^n\log_2(1-j2^{-n})
=n^2-\log_2(n!)+O(n^2 2^{-n}).
\]
Substituting this expansion into the preceding identity gives~\eqref{eq:single-capacity}, since $n^2 2^{-n}=O(n^{-99})$.
\end{proof}
\end{samepage}

\subsection{Proof of Theorem~\ref{thm:spec}}\label{sec:spec-proof}
For $H\in\mathcal A$, let $\mathcal A^H$ be the arrangement in $H$ consisting of the distinct nonempty proper intersections $H\cap K$, $K\in\mathcal A$. Double counting chamber--facet incidences gives the following identity.

\begin{lemma}[Chamber--facet incidence; {\cite[Lemma~5.1]{Anthony}}]\label{lem:facet-incidence}
For a finite central arrangement $\mathcal A$ in $\R^d$ whose hyperplanes intersect only at the origin,
\[
\sum_C
\#\{\text{facets of }\overline C\}
=2\sum_{H\in\mathcal A}r(\mathcal A^H),
\]
where the first sum is over all chambers $C$ of $\mathcal A$.
\end{lemma}

\begin{proof}[Proof of Theorem~\ref{thm:spec}]
The vectors in $E_n$ span $\R^{n+1}$, so the hyperplanes of $\mathcal A_n$ intersect only at the origin. By the identification of the specification number with the number of facets~\cite[Theorem~3.3]{Anthony} and Lemma~\ref{lem:facet-incidence},
\begin{equation}\label{eq:spec-incidence}
\sigma_n=\frac{2}{T_n}\sum_{e\in E_n}r(\mathcal A_n^{e^\perp}).
\end{equation}
The upper bound $\sigma_n\le2n$ is the Fukuda--Tamura--Tokuyama bound, in the form stated in~\cite[Theorem~4.14]{Anthony}.

Fix $e\in E_n\setminus\{e_0\}$. The restricted central arrangement $\mathcal A_n^{e^\perp}$ contains $e^\perp\cap e_0^\perp$. Taking its affine section at $e_0\cdot z=1$ and pairing opposite regions gives
\[
r(\mathcal A_n^{e^\perp})=2r(\mathcal D_n^{P_0\cap e^\perp}).
\]
For each subset $S$ in Lemma~\ref{lem:intersections} that contains $e$, the nonempty intersection $P_0\cap\bigcap_{f\in S}f^\perp$ is an intersection subspace of $\mathcal D_n^{P_0\cap e^\perp}$, obtained from the hyperplanes indexed by $S\setminus\{e\}$. Distinct such subsets give distinct intersections, so Lemma~\ref{lem:arrangement-bounds} bounds the number of regions below by the number of these subsets. When we sum over $e\ne e_0$, each $S$ contributes once for each of its $n$ elements. Combining this multiplicity with the factor two in the preceding identity gives
\begin{equation}\label{eq:spec-restricted-lower}
\sum_{e\in E_n\setminus\{e_0\}}r(\mathcal A_n^{e^\perp})
\ge2n\binom{2^n-1}{n}(1-Cn^{-99}).
\end{equation}
Substituting~\eqref{eq:spec-restricted-lower} into~\eqref{eq:spec-incidence} and using the classical upper bound on $T_n$ yields
\[
\sigma_n\ge\frac{4n}{T_n}\binom{2^n-1}{n}(1-Cn^{-99})
\ge2n-Cn^{-98}.\qedhere
\]
\end{proof}

\begin{samepage}
\subsection{Proof of Theorem~\ref{thm:span}}\label{sec:span-proof}
\subsubsection{Integral relations}\label{sec:relations}

Let $v_1,\ldots,v_n$ be independent uniform vectors in $\{-1,1\}^n$, and put
\[
R=(v_1,\ldots,v_{n-1}),\qquad M=(R,v_n).
\]
\end{samepage}
For any real nonzero vector $a$, define
\begin{equation}\label{eq:concentration}
p(a)=\mathbb P\{\epsilon^{\tran}a=0\},
\end{equation}
where $\epsilon$ is a uniform sign vector of the same dimension as $a$. Write $\supp(a)$ for the number of nonzero coordinates of $a$. For $a\in\R^n\setminus\{0\}$, write $E_a=\{Ma=0\}$. Since each row annihilates $a$ with probability $p(a)$ and the rows are independent, $\mathbb P(E_a)=p(a)^n$.

The maximal minors of $R$ are the determinants $D_i$ obtained by deleting row~$i$. Subtracting the first row from every other row of each minor shows that $2^{n-2}\mid D_i$. Divide out this common factor and define
\begin{equation}\label{eq:delta}
\Delta(R)=\frac{\gcd(D_1,\ldots,D_n)}{2^{n-2}}>0
\end{equation}
when $\rank R=n-1$. If $\rank R<n-1$, set $\Delta(R)=\infty$. Let $F_n$ be the event that the conclusion of Theorem~\ref{thm:span} fails for $r=n-1$.

\begin{proposition}[Reduction to a common divisor]\label{prop:reduction}
There are absolute constants $h,c,C>0$ such that
\begin{equation}\label{eq:reduction}
\mathbb P(F_n)\le2\mathbb P\{\Delta(R)>e^{hn}\}+Ce^{-cn}.
\end{equation}
\end{proposition}

\begin{lemma}[Subspace concentration bound; {\cite[Lemma~2]{KKS}}]\label{lem:subspace-concentration}
Let $Y$ be an $m\times q$ matrix with independent uniform sign entries. Suppose $S\subset\R^q\setminus\{0\}$ lies in a linear subspace of dimension $d$, where $1\le d\le m+1$, and $p(a)\le u$ for every $a\in S$. Then
\begin{equation}\label{eq:subspace-concentration}
\mathbb P\{Ya=0\text{ for some }a\in S\}
\le\binom m{d-1}u^{m-d+1}.
\end{equation}
\end{lemma}

\begin{proof}
Restrict the rows of $Y$, viewed as linear functionals, to a subspace of dimension $d$ containing $S$. If their common kernel meets $S$, the restricted rows span a space of dimension at most $d-1$. Hence some $d-1$ rows span all the restrictions, with redundant rows added if necessary. For each fixed set of $d-1$ row indices, choose a vector $a\in S$ annihilated by those rows whenever one exists, using only the selected rows. If their restrictions span all the restricted rows, every remaining row must also annihilate this $a$. The remaining rows are independent of the selected rows, so the conditional probability that they all annihilate $a$ is at most $u^{m-d+1}$. A union bound over the $\binom m{d-1}$ choices proves the claim.
\end{proof}

\begin{lemma}[Large coefficients in integer relations]\label{lem:large-coefficients}
There are absolute constants $h,c,C>0$ such that, with probability at least $1-Ce^{-cn}$ over $R$, every relation
\begin{equation}\label{eq:integer-relation}
\sum_{i<n}a_i v_i+a_n v=0,
\qquad a\in\Z^n,\quad a_n\ne0,
\end{equation}
with $v\in \{-1,1\}^n\setminus\{\pm v_i:i<n\}$ has more than $n/2$ coordinates satisfying $|a_i|>e^{hn}$.
\end{lemma}

\begin{proof}
We bound the probability that there exists a sign vector $v\notin\{\pm v_i:i<n\}$ in the span admitting such an integer relation with at most $n/2$ coordinates satisfying $|a_i|>e^{hn}$. Conditional on $R$, if such a vector exists, the independent uniform column $v_n$ is one of these vectors with probability at least $2^{-n}$. The probability to be bounded is therefore at most $2^n$ times the probability of the corresponding union of the events $E_a$.

Take $\gamma=1/8$ in~\cite[Corollary~2]{KKS}. There are constants $s_0\ge4$ and $\eta>0$ such that
\[
\mathbb P\left(\bigcup\{E_a:\supp(a)>s_0,\ p(a)>e^{-\eta n}\}\right)
\le8^{-n}.
\]
After multiplication by $2^n$, this class contributes at most $4^{-n}$ to the probability under consideration. Relations of support one are impossible, and those of support two with $a_n\ne0$ give only $v=\pm v_i$. For $3\le s\le s_0$, the Littlewood--Offord inequality~\cite{KKS,Odlyzko} gives $p(a)\le3/8$. Applying Lemma~\ref{lem:subspace-concentration} with $d=s$ on each support of size $s$ containing $n$, and summing over these supports, bounds the total contribution by
\begin{equation}\label{eq:small-support}
2^n\sum_{s=3}^{s_0}\binom{n-1}{s-1}\binom n{s-1}
(3/8)^{n-s+1}
\le Cn^{2s_0}(3/4)^n.
\end{equation}
The first binomial coefficient counts supports, and the second counts rows.

It remains to consider $p(a)\le e^{-\eta n}$. Put $m=\lceil n/2\rceil$ and $h=\eta/2$. Fix $I\subset[n]$ of size $m$ and integers $b_i$, $i\in I$, with $|b_i|\le e^{hn}$. The set
\[
S=\{a\in\R^n\setminus\{0\}:a_i=b_i\ (i\in I),\ p(a)\le e^{-\eta n}\}
\]
lies in a translate of a subspace of dimension $n-m$, so its linear span has dimension at most $n-m+1$. Lemma~\ref{lem:subspace-concentration} therefore gives
\[
\mathbb P\left(\bigcup_{a\in S}E_a\right)\le2^n e^{-\eta nm}.
\]
Summing over the choices of $I$ and the at most $(3e^{hn})^m$ coefficient assignments, and multiplying by the initial factor $2^n$, gives
\[
4^n\binom nm(3e^{hn})^m e^{-\eta nm}
\le\exp\left(-\frac12\eta nm+Cn\right).
\]
A vector with at most $n/2$ large coordinates has at least $m$ small coordinates, so every remaining violation is included. Together with~\eqref{eq:small-support} and $4^{-n}$, this proves the lemma.
\end{proof}

\begin{lemma}[Divisibility of maximal minors]\label{lem:divisibility}
Suppose $\rank R=n-1$, $v\in \{-1,1\}^n$, and
\[
\sum_{i<n}a_i v_i+a_n v=0,
\qquad a\in\Z^n,\quad a_n\ne0,\quad \gcd(a_1,\ldots,a_n)=1.
\]
Then $|a_n|\mid\Delta(R)$.
\end{lemma}

\begin{proof}
In the minor defining $D_i$, replace column~$j$ by the corresponding coordinates of $v$, and call the resulting determinant $D_{i,j}$. Multilinearity and subtraction of the first row from the other rows of the matrix defining $D_{i,j}$ give, respectively,
\[
a_nD_{i,j}=-a_jD_i,\qquad 2^{n-2}\mid D_{i,j}.
\]
Thus $a_n$ divides $a_jD_i/2^{n-2}$ for every $j<n$. It also divides $a_nD_i/2^{n-2}$. Since the coordinates of $a$ have greatest common divisor one, B\'ezout's identity implies that $a_n\mid D_i/2^{n-2}$ for every~$i$. Taking the greatest common divisor proves the assertion.
\end{proof}

\begin{proof}[Proof of Proposition~\ref{prop:reduction}]
Take $h$ from Lemma~\ref{lem:large-coefficients}. When $\rank M=n-1$, choose a primitive integer null vector $a$, normalized by $\gcd(a_1,\ldots,a_n)=1$. Its absolute coordinates are uniquely determined because the nullspace is one-dimensional. Define
\begin{equation}\label{eq:weight}
w(M)=\frac{2^n}{|\langle v_1,\ldots,v_n\rangle\cap \{-1,1\}^n|-2(n-1)}
\end{equation}
if $\rank M=n-1$ and more than $n/2$ coordinates of $a$ satisfy $|a_i|>e^{hn}$. Otherwise, set $w(M)=0$. Under the conditions for using this fraction, $\supp(a)>2$ for all sufficiently large $n$. Since the nullspace is one-dimensional, no two columns can be equal up to sign. Their span therefore contains at least $2n$ distinct cube vertices, which makes the denominator positive. Expressions involving $a$ below are interpreted as zero when $w(M)=0$.

Suppose $\rank R=n-1$, the event $F_n$ occurs, and the conclusion of Lemma~\ref{lem:large-coefficients} holds. Every additional last column has positive weight, and every other last column has weight zero. There are
\[
|\langle v_1,\ldots,v_{n-1}\rangle\cap \{-1,1\}^n|-2(n-1)
\]
additional columns, each sampled with probability $2^{-n}$. Thus $\E(w(M)\mid R)=1$ on this event, and
\begin{equation}\label{eq:weight-bound}
\mathbb P(F_n)\le\mathbb P\{\rank R<n-1\}+Ce^{-cn}+\E w(M).
\end{equation}
Both $w(M)$ and the distribution of $M$ are invariant under column permutations. Since more than half the coordinates are large when $w(M)>0$,
\begin{align}\label{eq:permutations}
\E w(M)
&\le\frac2n\sum_{i=1}^n\E\bigl[w(M)\mathbf 1_{\{|a_i|>e^{hn}\}}\bigr]\notag\\
&=2\E\bigl[w(M)\mathbf 1_{\{|a_n|>e^{hn}\}}\bigr].
\end{align}
If the random variable in the last expectation is nonzero, then $M$ has rank $n-1$ and $a_n\ne0$. Thus $R$ has full column rank and $v_n$ is additional, so Lemma~\ref{lem:divisibility} gives $|a_n|\mid\Delta(R)$. For fixed $R$ of full column rank with additional cube vertices, averaging over $v_n$ with the weight $w(M)$ gives the fraction of these vertices for which $w(M)>0$ and $|a_n|>e^{hn}$. This fraction is at most one and vanishes unless $\Delta(R)>e^{hn}$. If $R$ does not have full column rank or has no additional cube vertices, the conditional expectation is zero. Therefore
\begin{equation}\label{eq:conditional-weight}
\E\bigl[w(M)\mathbf 1_{\{|a_n|>e^{hn}\}}\mid R\bigr]
\le\mathbf 1_{\{\rank R=n-1,\ \Delta(R)>e^{hn}\}}.
\end{equation}
Combining~\eqref{eq:weight-bound}--\eqref{eq:conditional-weight} proves~\eqref{eq:reduction}, because the convention $\Delta(R)=\infty$ includes the case where $R$ does not have full column rank in the event on its right side.
\end{proof}

\subsubsection{Integral cokernels}\label{sec:cokernels}
\begin{proposition}[Cokernel bound]\label{prop:cokernel}
Let $B$ have independent uniform $0/1$ entries and size $N\times(N+1)$, and let $\Cok(B)=\Z^N/B\Z^{N+1}$. Set $\tau=|\Cok(B)|$, with $\tau=\infty$ when the cokernel is not finite. For every fixed $b>0$ and all sufficiently large $N$, with probability at least $1-C_bN^{-99}$, the cokernel is finite and
\begin{equation}\label{eq:cokernel-bound}
\tau\le e^{bN}.
\end{equation}
\end{proposition}

When $\tau$ is finite, write $\tau=\prod_p p^{\nu_p}$. On the event $\tau=\infty$, define $\nu_p=0$ for every prime $p$. For $x\ge2$, define
\[
\tau_{\le x}=\prod_{p\le x}p^{\nu_p},\qquad
\log\tau_{\le x}=\sum_{p\le x}\nu_p\log p.
\]
For finite $\tau$, $p\mid\tau$ if and only if $\rank_{\mathbb F_p}B<N$. We write $v_p$ for the exponent of $p$ in an integer, with $v_p(0)=\infty$. By Hadamard's inequality,
\begin{equation}\label{eq:valuation-cap}
\nu_p\le D_p:=\frac{CN\log N}{\log p}.
\end{equation}

\begin{lemma}[Moments for small primes]\label{lem:small-primes}
For every fixed positive integer $k$ and every fixed $K>0$,
\begin{equation}\label{eq:moment}
\E\bigl(\log\tau_{\le N^K}\bigr)^k
\le C_{k,K}(\log N)^k.
\end{equation}
\end{lemma}

\begin{proof}
\emph{Step 1. Joint divisibility.}
For fixed $s$, distinct primes $p_1,\ldots,p_s$, and integers $1\le j_i\le D_{p_i}$, we first show that
\begin{equation}\label{eq:joint-divisibility}
\mathbb P\{\nu_{p_i}\ge j_i\text{ for all }i\}
\le2^s\prod_{i=1}^s p_i^{-j_i}+C_s e^{-c_sN}.
\end{equation}
Append an independent uniform $0/1$ row to $B$ to obtain a square matrix $A$. When $\tau<\infty$, its order divides every maximal minor of $B$, so expansion along the new row gives $\tau\mid\det A$. Thus joint divisibility of $\tau$ implies the corresponding divisibility of $\det A$, and it suffices to bound the latter event, including the case $\det A=0$.

For a modulus that is a product of at most $s$ of the prime powers $p_i^{j_i}$, compare $A$ with a matrix having independent uniform entries modulo that modulus. Maples's comparison theorem~\cite[Theorem~1.2]{Maples} bounds the difference between the probabilities of each cokernel isomorphism class by $C_se^{-c_sN}$. The constants are uniform in the primes and exponents.

Apply inclusion--exclusion to the complementary events $v_{p_i}(\det A)<j_i$ to estimate joint divisibility. Each such event is equivalent to the cokernel modulo $p_i^{j_i}$ having order less than $p_i^{j_i}$. Finite abelian groups of order $p_i^\ell$ are indexed by partitions of $\ell$. The partition bound and~\eqref{eq:valuation-cap} therefore bound the number of isomorphism classes in each inclusion--exclusion term by
\[
\prod_{i=1}^s\sum_{0\le\ell<j_i}\exp(C\sqrt\ell)
\le\exp\bigl(C_s\sqrt{N\log N}\bigr).
\]
Since $s$ is fixed, summing the comparison errors over these classes and the at most $2^s$ terms still gives $C_se^{-c_sN}$.

For the uniform reference matrix, the Chinese remainder theorem makes the prime components independent. If $Y$ is a square matrix of order $N+1$ with independent entries distributed according to Haar measure on $\Z_p$, elimination gives
\[
v_p(\det Y)\ \overset{d}=\ Z_1+\cdots+Z_{N+1},
\qquad \mathbb P\{Z_i\ge j\}=p^{-ij},
\]
with independent $Z_i$. Indeed, choosing a pivot of minimum valuation in the first column and eliminating that column leaves a smaller matrix with independent entries distributed according to Haar measure, to which the same argument applies. To bound the tail of the sum, observe that
\[
\E p^{Z_2+\cdots+Z_{N+1}}
=\prod_{i=2}^{N+1}\frac{1-p^{-i}}{1-p^{1-i}}
=\frac{1-p^{-(N+1)}}{1-p^{-1}}\le2.
\]
Using the geometric tail of $Z_1$ and conditioning on $Z_2+\cdots+Z_{N+1}$ gives
\[
\mathbb P\{v_p(\det Y)\ge j\}\le2p^{-j}.
\]
Multiplying these bounds over the independent prime components and adding the comparison error proves~\eqref{eq:joint-divisibility}.

\smallskip
\emph{Step 2. The moment bound.}
Expand the moment in~\eqref{eq:moment}, grouping equal primes. For positive integers $m_1+\cdots+m_s=k$, the identity
\[
\nu^m=\sum_{j\ge1}\bigl(j^m-(j-1)^m\bigr)\mathbf 1_{\{\nu\ge j\}}
\]
and~\eqref{eq:joint-divisibility} yield
\begin{equation}\label{eq:mixed-moment}
\E\prod_{i=1}^s\nu_{p_i}^{m_i}
\le C_k\prod_{i=1}^s p_i^{-1}
+C_k(N\log N)^k e^{-c_kN}.
\end{equation}
Here we used $\sum_{j\ge1}(j^m-(j-1)^m)p^{-j}\le C_m/p$. The elementary prime estimate
\[
\sum_{p\le x}\frac{(\log p)^m}{p}\le C_m(\log x)^m,
\]
which follows by partial summation from $\sum_{p\le x}\log p\le Cx$, bounds the sum of the main terms by $C_{k,K}(\log N)^k$. Since $k,K$ are fixed, the error terms in~\eqref{eq:mixed-moment}, including their logarithmic factors, have a total contribution that is still exponentially small. Combining these two estimates proves~\eqref{eq:moment}.
\end{proof}

\begin{lemma}[Intermediate prime factors]\label{lem:medium-primes}
There is an absolute $\delta>0$ such that, for every fixed $K>0$ and all sufficiently large $N$, with probability at least $1-C_KN^{-K}-Ce^{-cN}$, the matrix $B$ has full row rank over every $\mathbb F_p$ with $N^K<p\le e^{\delta N}$.
\end{lemma}

\begin{proof}
Applying Nguyen--Wood's rank estimate~\cite[Proposition~2.2]{NW} with one excess column gives
\begin{equation}\label{eq:medium-prime}
\mathbb P\{\rank_{\mathbb F_p}B<N\}\le Cp^{-2}+Ce^{-cN}.
\end{equation}
Fix $0<\delta<c/2$. A union bound gives
\[
\mathbb P\{\rank_{\mathbb F_p}B<N\text{ for some }N^K<p\le e^{\delta N}\}
\le C\sum_{p>N^K}p^{-2}+Ce^{\delta N-cN}
\le C_KN^{-K}+Ce^{-c'N}.\qedhere
\]
\end{proof}

\subsubsection{Large prime factors}\label{sec:large-primes}
\begin{lemma}[Large primes]\label{lem:large-primes}
For every fixed $\delta>0$ and all sufficiently large $N$, with probability at least $1-C_\delta N^{-99}$, the matrix $B$ has full row rank over every $\mathbb F_p$ with $p>e^{\delta N}$.
\end{lemma}

Set
\begin{equation}\label{eq:prime-parameters}
t=\lceil2N^{1/4}\rceil,\qquad m=N-t,\qquad
\log p_0=\sqrt N+O(1),
\end{equation}
where $p_0$ is a prime, chosen by Bertrand's postulate. These parameters satisfy
\[
t\log N=o(\log p_0),\qquad t\log p_0=o(N).
\]
For $a\in\mathbb F_p^N$, put
\[
\rho_p(a)=\max_{z\in\mathbb F_p}\mathbb P\{X^{\tran}a=z\},
\qquad X\text{ uniform on }\{0,1\}^N.
\]

\begin{lemma}[A new column]\label{lem:new-column}
Fix $K,\delta>0$, and let $U$ consist of the first $m$ columns of $B$. With probability at least $1-C_{K,\delta}e^{-c_{K,\delta}N}$ over $U$, every prime $p>e^{\delta N}$ and every proper subspace $V\subset\mathbb F_p^N$ containing the columns of $U\bmod p$ satisfy
\begin{equation}\label{eq:new-column}
\mathbb P\{X\bmod p\in V\}\le N^{-K},
\end{equation}
where $X$ is an independent uniform $0/1$ column.
\end{lemma}

\begin{proof}
\emph{Step 1. Additive structure.}
Nguyen--Wood's inverse Littlewood--Offord theorem~\cite[Theorem~7.3]{NW}, applied to the uniform $0/1$ distribution with $\varepsilon=1/2$, gives the following structural statement. The subsequent reduction and counting steps adapt their approach in Sections~6--7. For fixed $D>0$, if $p>C_DN^D$ and $\rho_p(a)\ge N^{-D}$, then all but $t$ coordinates of $a$ lie in a proper symmetric generalized arithmetic progression
\[
Q=\left\{\sum_{i=1}^d z_i g_i:|z_i|\le L_i,\ z_i\in\Z\right\},
\qquad d\le C_D,
\]
with
\begin{equation}\label{eq:gap-volume}
|Q|=\prod_{i=1}^d(2L_i+1)
\le\max\{1,C_D\rho_p(a)^{-1}t^{-d/2}\}.
\end{equation}
Remove the generators for which $L_i=0$. If $a$ has more than $t$ nonzero coordinates, then $Q\ne\{0\}$, and therefore $d\ge1$, $|Q|\ge3$, and
\begin{equation}\label{eq:gap-improvement}
|Q|\le C_D\rho_p(a)^{-1}t^{-1/2}.
\end{equation}

\smallskip
\emph{Step 2. Reduction to one prime.}
Suppose~\eqref{eq:new-column} fails for some $p>e^{\delta N}$. A nonzero normal $a$ to $V$ satisfies $U^{\tran}a=0$ and $\rho_p(a)>N^{-K}$. By~\eqref{eq:gap-volume}, we can write $a=Tg$, where $T$ is an integer matrix with $N$ rows, at most $C_K+t$ columns, and entries bounded by $C_KN^K$. The progression columns of $T$ record the integer coefficients $z_i$, whereas its remaining columns are coordinate vectors for the exceptional entries. The vector $g$ consists of the progression generators and those exceptional values. Since $Tg=a\ne0$ and $U^{\tran}Tg=0$,
\begin{equation}\label{eq:rank-transfer}
\rank_{\mathbb F_p}(U^{\tran}T)<\rank_{\mathbb F_p}T.
\end{equation}
All minors of $T$ and $U^{\tran}T$ have absolute value at most
\[
\exp(C_KN^{1/4}\log N)<\min\{p,p_0\}
\]
for sufficiently large $N$. Consequently, a minor is nonzero over $\Q$ if and only if it is nonzero modulo either prime, so both matrices have the same rank over $\Q$, $\mathbb F_p$, and $\mathbb F_{p_0}$. The rank inequality~\eqref{eq:rank-transfer} therefore gives a nonzero
\[
a'\in\im(T\bmod p_0),\qquad U^{\tran}a'=0\pmod{p_0}.
\]
All but $t$ coordinates of $a'$ lie in a symmetric progression of rank at most $C_K$ and volume at most $C_KN^K$. This progression need not be proper, but its defining box retains the same volume bound.

\smallskip
\emph{Step 3. Counting at the fixed prime.}
The transfer in Step~2 preserves additive structure but gives no lower bound on $\rho_{p_0}(a')$. We therefore count the transferred vectors according to their concentration. First, for bounded progression rank $d$ and a numerical volume bound $V\ge1$, the number of vectors with $t$ exceptional coordinates and all other coordinates in the progression is at most
\begin{equation}\label{eq:gap-count}
\binom Nt p_0^{t+d}V^{m+d}.
\end{equation}
Choose the exceptional positions in $\binom Nt$ ways. There are then at most $V^d$ choices of lengths, $p_0^d$ choices of generators, $p_0^t$ choices of exceptional values, and $V^m$ choices of the other entries. Vectors with fewer exceptions are included by enlarging their exceptional set to size $t$. For fixed $d$ and $V\le N^{O(1)}$, the choices of $t$ and $p_0$ give
\begin{equation}\label{eq:subexponential-count}
\binom Nt p_0^{t+d}V^d=\exp(o(N)).
\end{equation}
For each fixed nonzero $a'$, independence of the columns gives
\[
\mathbb P\{U^{\tran}a'=0\}\le\rho_{p_0}(a')^m.
\]
Nonzero vectors supported on at most $t$ coordinates contribute at most
\[
(t+1)N^t p_0^t2^{-m}\le e^{-cN}.
\]
Here $\rho_{p_0}(a')\le1/2$ for every nonzero $a'$, by conditioning on all but one coordinate with a nonzero coefficient.

Choose a fixed $D>K+2$. For the vectors $a'$ obtained in Step~2 with $\rho_{p_0}(a')\le N^{-D}$, equations~\eqref{eq:gap-count}--\eqref{eq:subexponential-count}, with $V=C_KN^K$, give total probability at most
\[
e^{C_KN}N^{Km}N^{-Dm}\le e^{-c_KN\log N}.
\]
For the remaining vectors, divide $N^{-D}<\rho_{p_0}(a')\le1/2$ into $O_D(\log N)$ dyadic intervals $[q,2q]$. Apply~\eqref{eq:gap-volume} again, now with exponent $D$. Vectors supported on at most $t$ coordinates have already been treated, so~\eqref{eq:gap-improvement} gives $V\le C_Dq^{-1}t^{-1/2}$. For each interval, equations~\eqref{eq:gap-count}--\eqref{eq:subexponential-count} give
\[
e^{C_DN}(q^{-1}t^{-1/2})^m(2q)^m
\le e^{C_D'N}t^{-m/2}\le e^{-c_DN\log N}.
\]
Summing over the dyadic intervals and adding the contributions from sparse vectors and from vectors with small concentration proves the lemma.
\end{proof}

\begin{proof}[Proof of Lemma~\ref{lem:large-primes}]
Write $X_1,\ldots,X_{N+1}$ for the columns of $B$. With probability at least $1-Ce^{-cN}$, the first $N$ columns are independent over $\Q$. This follows from~\cite[Theorem~1]{KKS} by normalizing the first row and column of a sign matrix of order $N+1$.

If $\rank_{\mathbb F_p}B<N$, at least two of its $N+1$ columns fail to increase the rank. If $j$ is the second such index, the preceding columns have rank $j-2$ modulo $p$. For each possible $j$, use a fixed rule to choose a nonzero minor $D$ of order $j-1$ from the preceding columns whenever they are independent over $\Q$. Every prime for which $j$ is the second failure makes the preceding columns have rank less than $j-1$ and therefore divides $D$. Hadamard's inequality gives $|D|\le N^{N/2}$ and therefore leaves at most $C_\delta\log N$ candidate primes larger than $e^{\delta N}$. Because this list depends only on $X_1,\ldots,X_{j-1}$, it is independent of $X_j$.

For $j\le m$, at a prime where the preceding rank is $j-2$, a vector in that span is determined by $j-2$ of its coordinates. Independence of the remaining coordinates gives
\[
\mathbb P\{X_j\text{ belongs to that span}\mid X_1,\ldots,X_{j-1}\}
\le2^{-(N-j+2)}.
\]
Summing over $j\le m$ and over the at most $C_\delta\log N$ possible primes for each $j$ gives at most $C_\delta(\log N)2^{-t}$. For $j>m$, the preceding columns include $U$. On the event in Lemma~\ref{lem:new-column}, their span is proper and contains $U\bmod p$, so that lemma bounds each conditional probability by $N^{-K}$. Summing over the $t+1$ remaining columns and the candidate primes, and adding the exceptional probabilities, gives
\begin{align}\label{eq:large-prime-sum}
&\mathbb P\{\rank_{\mathbb F_p}B<N\text{ for some prime }p>e^{\delta N}\}\notag\\
&\hspace{1cm}\le C_{K,\delta}e^{-c_{K,\delta}N}
+C_\delta\log N\bigl(2^{-t}+(t+1)N^{-K}\bigr).
\end{align}
Taking $K=102$ proves the lemma.
\end{proof}

\subsubsection{Completion of the proof}\label{sec:span-completion}
\begin{proof}[Proof of Proposition~\ref{prop:cokernel}]
Apply Lemma~\ref{lem:large-primes} with $\delta$ from Lemma~\ref{lem:medium-primes}. Outside a set of probability at most $CN^{-99}+C_KN^{-K}$, these lemmas give full row rank modulo every prime greater than $N^K$. It follows that
\[
\tau=\tau_{\le N^K}<\infty.
\]
Indeed, rank deficiency over $\Q$ would imply rank deficiency modulo every prime, and a prime factor of a finite $\tau$ would also force rank deficiency modulo that prime. It remains to control the size of $\tau_{\le N^K}$. Applying Markov's inequality to $\log\tau_{\le N^K}$ using Lemma~\ref{lem:small-primes}, and including the exceptional probability above, yields
\[
\mathbb P\{\tau>e^{bN}\}
\le CN^{-99}+C_KN^{-K}
+C_{k,K,b}\left(\frac{\log N}{N}\right)^k.
\]
Taking $k=K=102$ makes the right side $O_b(N^{-99})$ and proves the proposition.
\end{proof}

\begin{proof}[Proof of Theorem~\ref{thm:span}]
Change row and column signs so that the first row and column of $R$ are all~$1$. Conditional on their original values, the remaining entries are still independent uniform signs. The resulting matrix has the form
\[
\begin{pmatrix}
1&\mathbf 1^{\tran}\\
\mathbf 1&J-2B^{\tran}
\end{pmatrix},
\]
where $B$ has independent uniform $0/1$ entries and size $(n-2)\times(n-1)$, and $J$ and $\mathbf 1$ have all entries equal to one. Subtract the first row from every other row, and then subtract the first column from every other column. The result is $\diag(1,-2B^{\tran})$. These integral row and column operations preserve the greatest common divisor of maximal minors. The formulas relating Smith normal form to the cokernel and the greatest common divisors of minors~\cite[Theorems~2.3--2.4]{StanleySNF}, applied to $B^{\tran}$, give
\begin{equation}\label{eq:cokernel-identity}
\Delta(R)=|\Cok(B)|,\qquad
\Cok(B)=\Z^{n-2}/B\Z^{n-1}.
\end{equation}
Since $\rank R=1+\rank B$, the identity also holds with both sides set to infinity when the ranks are deficient.

Apply Proposition~\ref{prop:cokernel} to $B$ with $N=n-2$ and $b=h$. By~\eqref{eq:cokernel-identity} and $e^{hn}\ge e^{hN}$, this bounds the event on the right side of~\eqref{eq:reduction} and proves the theorem for $r=n-1$. For $r<n-1$, append independent sign vectors to obtain $n-1$ vectors. On the event that the full list is linearly independent and its span meets the cube only in its members and their negatives, linear independence excludes the appended vectors and their negatives from the span of the first $r$ vectors. Thus the same conclusion holds for the first $r$ vectors.
\end{proof}

\subsection{Proof of Theorem~\ref{thm:independence}}\label{sec:dependence}
\begin{lemma}[Equal or opposite columns]\label{lem:pairs}
Let $M$ be an $n\times r$ matrix with independent uniform sign entries, and let $F$ be the event that two columns are equal up to a sign. Then
\begin{equation}\label{eq:pair-probability}
\mathbb P(F)=2\binom r2\,2^{-n}+O(r^4 4^{-n}).
\end{equation}
\end{lemma}

\begin{proof}
Each of the $2\binom r2$ prescribed sign relations has probability $2^{-n}$. Two distinct relations have intersection probability at most $4^{-n}$. Inclusion--exclusion gives~\eqref{eq:pair-probability}.
\end{proof}

Following~\cite[Definition~3.1]{JSSsing}, let $\Cons(\delta,\rho)$ consist of the unit vectors $a\in\R^r$ such that, for some $\lambda\in\R$,
\[
|a_i-\lambda|\le\frac{\rho}{\sqrt r}
\quad\text{for at least }(1-\delta)r\text{ indices }i.
\]
As before, $p(a)=\mathbb P\{\epsilon^{\tran}a=0\}$ for a uniform sign vector $\epsilon$.

\begin{lemma}[Null vectors with large concentration]\label{lem:large-concentration}
Let $M$ be an $r\times r$ matrix with independent uniform sign entries. Then
\begin{equation}\label{eq:large-concentration}
\mathbb P\bigl\{Ma=0\text{ for some }a\ne0,\quad
\supp(a)>2,\quad p(a)>(2/3)^r\bigr\}
=O(2^{-r}e^{-cr}),
\end{equation}
where $c>0$ and the implied constant are absolute.
\end{lemma}

\begin{proof}
The estimate in the remark following~\cite[Theorem~1.1]{JSSrank}, with deficiency two, gives
\begin{equation}\label{eq:corank-two}
\mathbb P\{\rank M\le r-2\}
\le(1/2+o(1))^{2r}=O(2^{-r}e^{-cr}).
\end{equation}
It remains to consider rank $r-1$.

By~\cite[Theorem~1.5]{JSSsing}, there are $\delta,\rho>0$ for which
\[
\mathbb P\{Ma=0\text{ for some }a\in\Cons(\delta,\rho)\}
\le 2\binom r2\,2^{-r}+O(2^{-r}e^{-cr}).
\]
For large $r$, every unit vector of support two belongs to $\Cons(\delta,\rho)$. At rank $r-1$, the nullspace is one-dimensional, so a relation of support two excludes any relation with larger support. Subtracting the pair probability~\eqref{eq:pair-probability} with $n=r$ therefore isolates the nontrivial almost-constant relations, up to the exception for matrices of rank at most $r-2$ bounded in~\eqref{eq:corank-two}. This gives
\begin{equation}\label{eq:nontrivial-constant}
\mathbb P\left\{\begin{array}{c}
\rank M=r-1,\quad Ma=0\text{ for some }a\in\Cons(\delta,\rho),\\
\supp(a)>2
\end{array}\right\}
=O(2^{-r}e^{-cr}).
\end{equation}

\medskip
\noindent\emph{Symmetry under changes of column signs.}
Write the rows of $M$ as $r_1^{\tran},\ldots,r_r^{\tran}$. Let $A$ consist of the first $r-1$ rows, and choose a unit vector $a$ in its nullspace using $A$ alone. Proposition~3.7 of~\cite{JSSsing} supplies constants $\gamma,L>0$ such that, outside an event of probability at most $4^{-r}$ depending only on $A$, either $a\in\Cons(\delta,\rho)$ or
\begin{equation}\label{eq:balanced-row}
\mathbb P\{r_r^{\tran}a=0\mid A,\ r_r\in W_\gamma\}
\le L(3/5)^r,
\qquad
W_\gamma=\left\{\varepsilon\in \{-1,1\}^r:\left|\sum_{i=1}^r\varepsilon_i\right|\le2\gamma r\right\}.
\end{equation}
For uniform signs, taking the parameter in that proposition to be $\log(6/5)$ gives the threshold $(3/5)^r$. The definitions in~\cite[Definitions~2.2 and~3.6]{JSSsing} convert this threshold bound into the bound for the atom at zero in~\eqref{eq:balanced-row}. To control the cost of conditioning, we also use $\mathbb P\{r_r\in W_\gamma\}\ge1/2$ for sufficiently large $r$, which follows from Chebyshev's inequality.

Set
\[
\mathcal J=\{\rank A=r-1,\quad\supp(a)>2,\quad
p(a)>(2/3)^r,\quad r_r^{\tran}a=0\}.
\]
On $\{\rank A=r-1\}$, the vector $a$ is unique up to sign, so changing column signs preserves $\mathcal J$. These transformations also preserve the law of $M$ and act transitively on the possible last rows. The conditional probability of $\mathcal J$ is therefore the same for every value of the last row. This uses invariance of $\mathcal J$, not of the class of almost-constant vectors, and in particular gives
\begin{equation}\label{eq:sign-invariance}
\mathbb P(\mathcal J)=\mathbb P(\mathcal J\mid r_r\in W_\gamma).
\end{equation}
When $p(a)>(2/3)^r$, the right side of~\eqref{eq:balanced-row} is at most $L(9/10)^r p(a)$. Use~\eqref{eq:sign-invariance} and separate the almost-constant case and the exceptional event in~\eqref{eq:balanced-row}. Averaging over $A$, which is independent of the last row, then bounds the remaining contribution by $L(9/10)^r\mathbb P(\mathcal J)$. Since the conditioning event has probability at least $1/2$, the almost-constant contribution costs at most a factor two, giving
\begin{align*}
\mathbb P(\mathcal J)
&\le 2\mathbb P\{\mathcal J,\ a\in\Cons(\delta,\rho)\}
    +4^{-r}+L(9/10)^r\mathbb P(\mathcal J)\\
&\le O(2^{-r}e^{-cr})+L(9/10)^r\mathbb P(\mathcal J).
\end{align*}
The second line follows from~\eqref{eq:nontrivial-constant} because the full matrix has rank $r-1$ on $\mathcal J$. For sufficiently large $r$, $L(9/10)^r\le1/2$, so moving the last term to the left gives
\[
\mathbb P(\mathcal J)=O(2^{-r}e^{-cr}).
\]
Every matrix of rank $r-1$ has $r-1$ independent rows, so a union bound over the $r$ possible omitted rows covers this case. The factor $r$ is absorbed by decreasing $c$. Adding the contribution from matrices of rank at most $r-2$ in~\eqref{eq:corank-two} proves~\eqref{eq:large-concentration}.
\end{proof}

\begin{proof}[Proof of Theorem~\ref{thm:independence}]
Let $M=(v_1,\ldots,v_r)$ and let $F$ be the event in Lemma~\ref{lem:pairs}. Let $E$ be the event that $Ma=0$ for some $a$ with $\supp(a)>2$. Relations of support one are impossible, and those of support two are precisely the relations counted by $F$. Hence
\[
0\le\mathbb P\{\rank M<r\}-\mathbb P(F)\le\mathbb P(E).
\]
Adding columns preserves $E$ because any relation extends by assigning zero coefficients to the new columns. It therefore suffices to bound $\mathbb P(E)$ for $r=n-1$.

For this value of $r$, Lemma~\ref{lem:subspace-concentration}, with $d=r$ and $n$ rows, gives
\[
\mathbb P\bigl\{Ma=0\text{ for some }a\ne0,\quad p(a)\le(2/3)^r\bigr\}\le \binom n2(4/9)^r.
\]
Indeed, as in the proof of Lemma~\ref{lem:subspace-concentration}, a common null vector with concentration at most $(2/3)^r$ is chosen using only a fixed set of $r-1$ rows, whenever such a vector exists. The remaining two rows are independent of this choice and of each other, so the probability that both annihilate the vector is at most $(4/9)^r$. The union over the choices of rows accounts for $\binom n2$. Any relation with $\supp(a)>2$ and $p(a)>(2/3)^r$ must also annihilate the first $r$ rows, so Lemma~\ref{lem:large-concentration} bounds its occurrence probability. Combining the two cases gives
\[
\mathbb P(E)\le\binom n2(4/9)^r+O(2^{-r}e^{-cr})
=O(2^{-n}e^{-cn}),
\]
since $r=n-1$ and $4/9<1/2$. The monotonicity established above extends this bound to every $r\le n-1$. Combining it with~\eqref{eq:pair-probability} proves the theorem after absorbing $r^4 4^{-n}$ into the error term.
\end{proof}

\end{document}